\documentclass[english,letterpaper,11pt,reqno]{amsart}

\usepackage{appendix}
\usepackage{amsbsy}
\usepackage{amsfonts}
\usepackage{amsmath}
\usepackage{amssymb}
\usepackage{amsthm}
\usepackage{graphicx}
\usepackage{ifthen}
\usepackage{textcomp}
\usepackage{soul}

\usepackage{dsfont}
\usepackage[bookmarksnumbered,colorlinks]{hyperref}
\hypersetup{colorlinks=true, linkcolor=blue, citecolor=blue}
\newcommand{\lra}{\longrightarrow}

\newcommand{\vep}{\varepsilon}

\newtheorem{lemma}{Lemma}
\newtheorem{cor}{Corollary}

\newtheorem{defn}{Definition}
\newtheorem{prop}{Proposition}

\newcommand{\beqa}{\begin{eqnarray}}
\newcommand{\beq}{\begin{equation}}
\newcommand{\eeqa}{\end{eqnarray}}
\newcommand{\eeq}{\end{equation}}

\newcommand\ip[2]{\langle {#1},{#2}\rangle} 
\newcommand\vv[1]{{\boldsymbol {\it #1}}} 
\newcommand\ii{{\bf i}}

\newcommand\mb{\overline{\boldsymbol m}}
\newcommand\mm{{\boldsymbol m}}
\newcommand\kk{{\boldsymbol k}}
\newcommand\xx{\vv{x}}
\newcommand\yy{\vv{y}}
\newcommand\cd[2]{\nabla_{\!#1}{#2}}

\usepackage{enumitem, color, amssymb}

\begin{document}
\title[]{The Newman-Penrose Formalism for Riemannian 3-manifolds}
\author[]{Amir Babak Aazami}
\address{Kavli IPMU (WPI)\hfill\break\indent
The University of Tokyo\hfill\break\indent
Kashiwa, Chiba 277-8583, Japan}
\email{amir.aazami@ipmu.jp}

\thanks{This work was supported by the World Premier International Research Center Initiative (WPI), MEXT, Japan.}

\maketitle
\begin{abstract}
We adapt the Newman-Penrose formalism in general relativity to the setting of three-dimensional Riemannian geometry, and prove the following results.  Given a Riemannian 3-manifold without boundary and a smooth unit vector field $\kk$ with geodesic flow, if an integral curve of $\kk$ is hypersurface-orthogonal at a point, then it is so at every point along that curve.  Furthermore, if $\kk$ is complete, hypersurface-orthogonal, and satisfies $\text{Ric}(\kk,\kk) \geq 0$, then its divergence must be nonnegative.  As an application, we show that if the Riemannian 3-manifold is closed and a unit length $\kk$ with geodesic flow satisfies $\text{Ric}(\kk,\kk) > 0$, then $\kk$ cannot be hypersurface-orthogonal, thus recovering a result in \cite{hp13}.  Turning next to scalar curvature, we derive an evolution equation for the scalar curvature in terms of unit vector fields $\kk$ that satisfy the condition $R(\kk,\cdot,\cdot,\cdot) = 0$.  When the scalar curvature is a nonzero constant, we show that a hypersurface-orthogonal unit vector field $\kk$ satisfies $R(\kk,\cdot,\cdot,\cdot) = 0$ if and only if it is a Killing vector field.
\end{abstract}

\section{Introduction}
The goal of this paper is to examine geometric properties of vector flows\,---\,divergence, shear, hypersurface-othogonality, and the property of being geodesic\,---\,and to understand when these can be had in three-dimensional Riemannian manifolds.  We will be particularly interested in the relationship between the Ricci or scalar curvature of the manifold and the existence of vector flows with one or more of the properties above.  To explore this relationship, we work with a three-dimensional Riemannian version of the Newman-Penrose formalism \cite{newpen62}; ultimately, we have in mind results analogous to the well-known Sachs equations \cite{sachs1961} and the Goldberg-Sachs theorem \cite{gsrepub} from four-dimensional Lorentzian (spacetime) geometry.  We point out that a four-dimensional Riemannian version of the Goldberg-Sachs theorem (not employing the Newman-Penrose formalism) has already been undertaken in \cite{apostolov97}; moreover, in three dimensions, the Riemannian Newman-Penrose formalism used here is essentially identical to the one for stationary spacetimes studied in \cite{perjes} (see also \cite{hall87}).  Indeed, a secondary goal of this paper is to offer up the Newman-Penrose formalism as an avenue of pursuit in the study of vector flows in three-dimensional Riemannian geometry.

\vskip 12pt

The essence of the Newman-Penrose formalism is to express the covariant derivative, Lie bracket, Riemann curvature tensor, and the differential Bianchi identities in terms of coefficients that directly represent the properties of the flow mentioned above\,---\,the so-called \emph{spin coefficients}.  These spin coefficients then appear in first order differential equations, the \emph{generalized Sachs equations,} one of which is the well-known \emph{Raychaudhuri equation} from general relativity \cite{ray1955}.  Once these differential equations have been written down, our results can simply be ``read off\," from them, with very little work; indeed, that these are \emph{first order} differential equations is in fact one of the main virtues of the formalism.  Finally, given how well-known the Newman-Penrose formalism and the Goldberg-Sachs theorem are, this paper is more or less perfectly straightforward, and in fact we adapt the elegant (spacetime) treatment of these topics as they are presented in \cite[Chapter~5]{o1995}, retaining the same notation.  (Proposition \ref{prop:ray}, for example, is an identical analogue of two well-known results from general relativity.)  The main results of this paper are Propositions \ref{prop:ray} and \ref{prop:GS0}, followed, respectively, by Corollaries \ref{cor:hp} and \ref{cor:GS0}.  The author would like to thank Graham Cox for helpful discussions.

\section{Vector flows on Riemannian 3-manifolds}
\label{sec:review}
(This section parallels, but is not identical to, the spacetime treatment in \cite[p.~327-9]{o1995}.)  Let $\kk$ be a smooth unit vector field defined in an open subset of a Riemannian 3-manifold $(M,\ip{\,}{})$ without boundary, so that $\cd{\vv{v}}{\kk} \perp \kk$ for all vectors $\vv{v}$.  We are interested in the flow of such a vector field; in particular, in obtaining curvature conditions on $M$ under which the flow can be any combination of the following: \emph{geodesic, divergence-free, hypersurface-orthogonal,} or \emph{shear-free}.  To formalize the last two of these properties, let $\xx$ and $\yy$ be two smooth vector fields such that $\{\kk,\xx,\yy\}$ is a local orthonormal frame.  At each point $p$ in the domain of this frame, $\{\xx_p,\yy_p\}$ is an orthonormal basis for the orthogonal complement $\kk_p^{\perp}$, a subspace to which $\nabla$ descends as a linear map:
$$
\nabla\colon \kk_p^{\perp} \lra \kk_p^{\perp}\hspace{.2in},\hspace{.2in}\vv{v}_p\ \mapsto\ \cd{\vv{v}_p}{\kk}.
$$
We visualize each $\kk_p^{\perp}$ as a ``screen" showing cross sections of the flow of $\kk$, and the linear map $\nabla$ as approximating the evolution of this flow, in the following sense.  An expansion in normal coordinates $(x^i)$ centered at $p$ shows that for any nearby point $q = \text{exp}_p(x^i(q)\partial_i|_p)$,  
\beqa
\label{eqn:flow}
\kk_q\ \approx\ \Big(\kk^i(p)\ +\ (\nabla_{x(q)^j\,\partial_j|_p}\kk)^i\Big)\,\partial_i\big|_q,
\eeqa
so that the difference in the components of $\kk_p$ and $\kk_q$ is approximated by the linear map $\nabla$ (because $\kk$ has constant length, $\nabla$ is completely determined by its restriction to $\kk_p^{\perp}$).  Now consider a small disk $C$ of radius $\vep$ in the screen $\kk_p^{\perp}$ at $p$, given by $C = \{\vv{v}_p \in \kk_p^{\perp} : \ip{\vv{v}_p}{\vv{v}_p} \leq \vep\}$.  The linear map $\nabla$ sends the basis $\{\xx_p,\yy_p\}$ to $\{\cd{\xx_p}{\kk},\cd{\yy_p}{\kk}\}$, deforming the disk $C$; since the ``infinitesimal generators" $\kk_p$ themselves approximate the flow, \eqref{eqn:flow} allows us to interpret the deformation of $C$ via $\nabla$ as arising from the change in the flow of $\kk$.  It is in this sense that the linear map $\nabla$, whose matrix is
\beqa
\label{eqn:matrix}
\nabla\ =\ 
    \left[
      \begin{array}{cc}
        \ip{\cd{\xx_p}{\kk}}{\xx_p} & \ip{\cd{\yy_p}{\kk}}{\xx_p}\\
        \ip{\cd{\xx_p}{\kk}}{\yy_p} & \ip{\cd{\yy_p}{\kk}}{\yy_p}\\
      \end{array}
    \right],
\eeqa
approximates the evolution of the flow of $\kk$.  By abuse of notation, we'll use $\nabla$ to denote both the covariant differential and its matrix \eqref{eqn:matrix}.  Regarding the latter, note that since $\cd{\kk}{\kk} \perp \kk$,
$$
\text{div}\,\kk|_p\ =\ \ip{\cd{\xx_p}{\kk}}{\xx_p}\, +\, \ip{\cd{\yy_p}{\kk}}{\yy_p}\ =\ \text{tr}\,\nabla.
$$
(To say that $\kk$ is divergence-free is thus to say that $\text{tr}\,\nabla = 0$ at every point.)  Next, a simple application of Frobenius's theorem shows that $\kk$ is \emph{hypersurface-othogonal} at $p$ (that is, $\kk_p^{\perp}$ is integrable) if and only if the off-diagonal elements of \eqref{eqn:matrix} satisfy
\beqa
\label{eqn:rotation}
\ip{\cd{\yy_p}{\kk}}{\xx_p}\ -\ \ip{\cd{\xx_p}{\kk}}{\yy_p}\ :=\ \omega\ =\ 0.
\eeqa
Thus $\kk$ is hypersurface-othogonal if and only if $\kk$ is \emph{irrotational,} since the skew-symmetric part of $\nabla$ corresponds to an infinitesimal rotation of the disk $C$; note also that $\omega$ is invariant up to sign: its square is the determinant of the skew-symmetric part of $\nabla$.  Finally, we will also be interested in the trace-free symmetric part of \eqref{eqn:matrix}, whose components are usually written as a complex quantity, the \emph{complex shear:}
\beqa
\label{eqn:shear}
\sigma\ :=\ \frac{1}{2}\Big(\ip{\cd{\yy_p}{\kk}}{\yy_p} - \ip{\cd{\xx_p}{\kk}}{\xx_p}\Big)\ +\ \frac{\ii}{2}\Big(\ip{\cd{\yy_p}{\kk}}{\xx_p} + \ip{\cd{\xx_p}{\kk}}{\yy_p}\Big).
\eeqa
This quantifies the distortion of the circular disk $C$ into an elliptical region of the same area (the change in the area itself is given by the divergence of $\kk$); note that although $\sigma$ itself is not invariant, its magnitude $|\sigma^2|$ is (see \cite[p.~329-30]{o1995}).  We say that the flow of $\kk$ is \emph{shear-free} if $\sigma = 0$.  Recalling that a vector field $\vv{z}$ is a Killing vector field if and only if the map $\vv{v} \mapsto \cd{\vv{v}}{\vv{z}}$ is a skew-symmetric $(1,1)$-tensor, we immediately see that Killing vector fields on Riemannian 3-manifolds must be shear-free.  In fact, the following is true.  

\begin{lemma}
\label{lemma:Killing}
Let $M$ be a Riemannian 3-manifold and $\kk$ a smooth unit vector field defined in an open subset of  $M$.  The flow of $\kk$ is geodesic, divergence-free, and shear-free if and only if $\kk$ is a Killing vector field.
\end{lemma}

\begin{proof}
Let $\kk,\xx,\yy$ be a local orthonormal frame.  Then $\text{div}\,\kk = \text{Re}[\sigma] = 0$ yields $\ip{\cd{\xx}{\kk}}{\xx} = \ip{\cd{\yy}{\kk}}{\yy} = 0$, while $\cd{\kk}{\kk} = 0$ gives $\ip{\cd{\kk}{\kk}}{\xx} = \ip{\cd{\kk}{\kk}}{\yy} = 0$.  Together with $\text{Im}[\sigma] = \ip{\cd{\yy}{\kk}}{\xx} + \ip{\cd{\xx}{\kk}}{\yy} = 0$, it follows that for any smooth vector fields $\vv{v},\vv{w}$, 
$$
\ip{\cd{\vv{v}}{\kk}}{\vv{w}}\, +\, \ip{\vv{v}}{\cd{\vv{w}}{\kk}}\ =\ 0,
$$
once we write $\vv{v},\vv{w}$ in terms of $\kk,\xx,\yy$.  This is the Killing condition for $\kk$.  The converse follows similarly.
\end{proof}

Lemma \ref{lemma:Killing} already exemplifies how the properties of flows in which we are interested can impose curvature conditions on the manifold.  For example, Riemannian manifolds (of any dimension) satisfying ${\rm Ric} < 0$ do not admit constant length Killing vector fields (see, e.g., \cite{BN}).  Thus we may say, using our formalism, that in three dimensions such manifolds do not admit unit vector fields whose flow is simultaneously geodesic, divergence-free, and shear-free.  To obtain further results along these lines, we will use a three-dimensional Riemannian version of the Newman-Penrose formalism, to which we now turn.

\section{The Newman-Penrose formalism for Riemannian 3-manifolds}
\label{sec:NP}
Let $\{\kk,\xx,\yy\}$ be as above and define the complex-valued quantities
\beqa
\label{eqn:complex}
\mm\ :=\ \frac{1}{\sqrt{2}}(\xx-\ii \yy)\hspace{.2in},\hspace{.2in}\mb\ :=\ \frac{1}{\sqrt{2}}(\xx+\ii \yy),\nonumber
\eeqa
so that $\ip{\mm}{\mm} = \ip{\mb}{\mb} = \ip{\kk}{\mm} = \ip{\kk}{\mb} = 0$, while $\ip{\kk}{\kk} = \ip{\mm}{\mb} = 1$ (e.g., $\ip{\kk}{\mm} = \frac{1}{\sqrt{2}}(\ip{\kk}{\xx} - \ii\ip{\kk}{\yy})$; similarly with the others).  Henceforth we work with the \emph{complex triad} $\{\kk,\mm,\mb\}$ in place of $\{\kk,\xx,\yy\}$; the former is the (three-dimensional) Riemannian analogue of the complex null tetrad $\{\kk,\vv{l},\mm,\mb\}$ in the original Newman-Penrose formalism (see \cite[p.~332]{o1995}).  One fewer vector simplifies matters considerably, as there are now only five spin coefficients (as compared to the four-dimensional spacetime case, in which there are twelve spin coefficients).

\begin{defn}
\label{def:spin}
The {\rm spin coefficients} of the complex triad $\{\kk,\mm,\mb\}$ are the complex-valued functions
\beqa
\label{eqn:sc}
\kappa &=& -\ip{\cd{\kk}{\kk}}{\mm}\hspace{.2in},\hspace{.2in}\rho\ =\ -\ip{\cd{\mb}{\kk}}{\mm}\hspace{.2in},\hspace{.2in}\sigma\ =\ -\ip{\cd{\mm}{\kk}}{\mm},\nonumber\\
&&\hspace{.45in}\vep\ =\ \ip{\cd{\kk}{\mm}}{\mb}\hspace{.2in},\hspace{.2in}\beta\ =\ \ip{\cd{\mm}{\mm}}{\mb}.\nonumber
\eeqa
\end{defn}
Since $\{\kk,\xx,\yy\}$ is an orthonormal frame, we could just as well have written $\kappa = \ip{\kk}{\cd{\kk}{\mm}}$, where ${\cd{\kk}{\mm}} = \frac{1}{\sqrt{2}}(\cd{\kk}{\xx} - \ii \cd{\kk}{\yy})$; similarly with the others.  The spin coefficients are the objects of interest in the Newman-Penrose formalism.  To begin with, note that the flow of $\kk$ is geodesic, $\cd{\kk}{\kk} = 0$, if and only if $\kappa = 0$, because if $\cd{\kk}{\kk} \perp \mm$, hence to $\xx$ and $\yy$, then we must have $\cd{\kk}{\kk} = c\kk$, which $c$ must be zero because $\cd{\kk}{\kk} \perp \kk$ (by contrast, $c$ need not be zero if $\kk$ were a null vector field in a Lorentzian manifold).  Next, note that $\vep$ is purely imaginary, $\vep + \bar{\vep} = 0$ (in fact $\vep = \ii \ip{\cd{\kk}{\xx}}{\yy}$), and that if we had designated $\xx$ or $\yy$ to be parallel along the flow of $\kk$, then $\vep$ would vanish (conversely, if $\kappa = 0$, then $\vep = 0$ implies that both $\xx$ and $\yy$ must be parallel along the flow of $\kk$).  Finally, it is easy to verify that the spin coefficient $\sigma$ is precisely the complex shear \eqref{eqn:shear}, while $-2\rho = \text{div}\,\kk + \ii\, \omega$.  Thus the spin coefficients $\kappa,\rho,\sigma$ directly represent the geometric properties of the flow discussed in Section \ref{sec:review} above.  We now express the covariant derivatives and the Lie brackets of $\{\kk,\mm,\mb\}$ in terms of the spin coefficients in Definition \ref{def:spin} (cf. \cite[Corollary 5.8.2, p. 334]{o1995} and equations (9a),(9c) in \cite{hall87}).

\begin{lemma}
\label{lemma:cd}
The covariant derivatives of the complex triad $\{\kk,\mm,\mb\}$ are
\beqa
\cd{\kk}{\kk} &=& -\bar{\kappa}\,\mm - \kappa\,\mb,\nonumber\\
\cd{\mm}{\kk} &=& -\bar{\rho}\,\mm - \sigma\,\mb,\nonumber\\
\cd{\kk}{\mm} &=& \kappa\,\kk + \vep\,\mm,\nonumber\\
\cd{\mm}{\mm} &=& \sigma\,\kk + \beta\,\mm,\nonumber\\
\cd{\mm}{\mb} &=& \bar{\rho}\,\kk - \beta\,\mb,\nonumber
\eeqa
while their Lie brackets are
\beqa
[\kk,\mm] &=& \kappa\,\kk + (\vep + \bar{\rho})\,\mm + \sigma\,\mb,\label{eqn:lb0}\\ 
\,[\mm,\mb] &=& (\bar{\rho} - \rho)\,\kk + \bar{\beta}\,\mm - \beta\,\mb.\nonumber
\eeqa
All other covariant derivatives and Lie brackets are obtained by complex conjugating these.
\end{lemma}

\begin{proof}
All the equations are straightforward.  For example, 
$$
\cd{\kk}{\kk}\ =\ \underbrace{\ip{\cd{\kk}{\kk}}{\kk}}_{0}\kk\ +\ \underbrace{\ip{\cd{\kk}{\kk}}{\mb}}_{-\bar{\kappa}}\mm\ +\ \underbrace{\ip{\cd{\kk}{\kk}}{\mm}}_{-\kappa}\mb
$$
(note that the right-hand side is real, as it must be), while the Lie bracket $[\kk,\mm]$ is computed using $[\kk,\mm] = \cd{\kk}{\mm} - \cd{\mm}{\kk}$.
\end{proof}

Next, we consider the Riemann 4-tensor $R$; with respect to $\{\kk,\mm,\mb\}$, the trace on its first and last indices is
$$
\sum g^{ab} R(a,\cdot,\cdot,b)\ =\ R(\kk,\cdot,\cdot,\kk)\ +\ R(\mm,\cdot,\cdot,\mb)\ +\ R(\mb,\cdot,\cdot,\mm),
$$
which leads to the following identities involving the Ricci tensor:
$$
(*)~\left\{
\begin{array}{rcl}
\text{Ric}(\mm,\mm) &=& -R(\kk,\mm,\kk,\mm),\nonumber\\
\text{Ric}(\kk,\kk) &=& -2R(\kk,\mm,\kk,\mb),\nonumber\\
\text{Ric}(\kk,\mm) &=& -R(\kk,\mm,\mm,\mb),\nonumber\\
\text{Ric}(\mm,\mb) &=& \frac{1}{2}\text{Ric}(\kk,\kk) - R(\mb,\mm,\mm,\mb).\nonumber
\end{array}
\right.
$$

(E.g., $\text{Ric}(\kk,\mm) = \frac{1}{\sqrt{2}}\big(\text{Ric}(\kk,\xx) - \ii\,\text{Ric}(\kk,\yy)\big)$; similarly with the others.)  Then, expressing the Riemann 4-tensor as
\beqa
R(\vv{u},\vv{v},\vv{w},\vv{z}) &=& \ip{\cd{\vv{u}}{ \cd{\vv{v}}{\vv{w}}} - \cd{\vv{v}}{ \cd{\vv{u}}{\vv{w}}} - \cd{[\vv{u},\vv{v}]}{\vv{w}}}{\vv{z}}\nonumber\\
&=& \vv{u}\ip{\cd{\vv{v}}{\vv{w}}}{\vv{z}} - \ip{\cd{\vv{v}}{\vv{w}}}{\cd{\vv{u}}{\vv{z}}} - \vv{v}\ip{\cd{\vv{u}}{\vv{w}}}{\vv{z}}\label{eqn:Riem}\\
&& \hspace{1in} +\ \ip{\cd{\vv{u}}{\vv{w}}}{\cd{\vv{v}}{\vv{z}}} - \ip{\cd{[\vv{u},\vv{v}]}{\vv{w}}}{\vv{z}}\nonumber
\eeqa
leads to the following equations, the three-dimensional Riemannian analogues of the \emph{generalized Sachs equations} (cf. \cite{sachs1961}, \cite[Proposition 5.8.9, p. 339]{o1995}, and equations (12a)-(12e) in \cite{hall87}).

\begin{lemma}
\label{lemma:sachs1}
The complex triad $\{\kk,\mm,\mb\}$ satisfies
\beqa
\label{eqn:Sachs1}
\kk[\rho] - \mb[\kappa] &=& |\kappa|^2 + |\sigma|^2 + \rho^2 + \kappa\bar{\beta} + \frac{1}{2} {\rm Ric}(\kk,\kk),\label{eqn:S1}\\
\kk[\sigma] - \mm[\kappa] &=& \kappa^2 + 2\sigma\vep + \sigma(\rho + \bar{\rho}) - \kappa \beta +{\rm Ric}(\mm,\mm),\phantom{\frac{1}{2}}\label{eqn:S2}\\
\mm[\rho] - \mb[\sigma] &=& 2 \sigma\bar{\beta} + (\bar{\rho}-\rho)\kappa + {\rm Ric}(\kk,\mm),\phantom{\frac{1}{2}}\label{eqn:S3}\\
\kk[\beta] - \mm[\vep] &=& \sigma(\bar{\kappa} - \bar{\beta}) + \kappa (\vep - \bar{\rho}) + \beta(\vep + \bar{\rho}) - {\rm Ric}(\kk,\mm),\phantom{\frac{1}{2}}\nonumber\\
\mm[\bar{\beta}] + \mb[\beta] &=& |\sigma|^2 - |\rho|^2 -2|\beta|^2 + (\rho - \bar{\rho})\vep - {\rm Ric}(\mm,\mb) + \frac{1}{2} {\rm Ric}(\kk,\kk).\nonumber
\eeqa
\end{lemma}

\begin{proof}
The proof is straightforward.  The expression for $\kk[\rho] - \mb[\kappa]$ is obtained by using \eqref{eqn:Riem} and Lemma \ref{lemma:cd} to re-express $R(\kk,\mb,\kk,\mm)$.  Likewise, $\kk[\sigma] - \mm[\kappa]$ is obtained by considering $R(\kk,\mm,\kk,\mm)$,  $\mm[\rho] - \mb[\sigma]$ by considering $R(\mb,\mm,\kk,\mm)$, $\kk[\beta] - \mm[\vep]$ by considering $R(\kk,\mm,\mm,\mb)$, and $\mm[\bar{\beta}] + \mb[\beta]$ by considering $R(\mb,\mm,\mm,\mb)$.
\end{proof}

Two more equations result from the differential Bianchi identities (cf. \cite[Lemma 5.9.1, p. 341]{o1995} and equations (15a)-(15b) in \cite{hall87}).
\begin{lemma}
\label{lemma:bid}
The complex triad $\{\kk,\mm,\mb\}$ satisfies
\beqa
&&\hspace{-.4in}\kk[{\rm Ric}(\kk,\mm)]\, -\, \frac{1}{2}\mm[{\rm Ric}(\kk,\kk)]\ +\ \mb[{\rm Ric}(\mm,\mm)]\ =\ \label{eqn:bid}\\
&&\hspace{-.1in}\kappa\,{\rm Ric}(\kk,\kk)\ +\ \big(\vep + 2\rho + \bar{\rho}\big){\rm Ric}(\kk,\mm)\ +\ \sigma\,{\rm Ric}(\kk,\mb)\phantom{\frac{1}{2}}\nonumber\\
&&\hspace{1.2in} -\ \big(\bar{\kappa} + 2\bar{\beta}\big){\rm Ric}(\mm,\mm)\, -\, \kappa\,{\rm Ric}(\mm,\mb)\phantom{\frac{1}{2}}\nonumber\\
\text{and}&&\nonumber\\
&&\hspace{-.4in}\mm[{\rm Ric}(\kk,\mb)]\ +\ \mb[{\rm Ric}(\kk,\mm)]\, -\, \kk[{\rm Ric}(\mm,\mb) - \frac{1}{2}{\rm Ric}(\kk,\kk)]\ =\ \label{eqn:bid2}\\
&&\hspace{-.1in}(\rho+\bar{\rho})\big({\rm Ric}(\kk,\kk)-{\rm Ric}(\mm,\mb)\big)\, -\, \bar{\sigma}{\rm Ric}(\mm,\mm)\, -\, \sigma{\rm Ric}(\mb,\mb)\phantom{\frac{1}{2}}\nonumber\\
&&\hspace{1.2in} -\ \big(2\bar{\kappa} + \bar{\beta}){\rm Ric}(\kk,\mm)\, -\, \big(2\kappa + \beta){\rm Ric}(\kk,\mb).\phantom{\frac{1}{2}}\nonumber
\eeqa
\end{lemma}

\begin{proof}
For \eqref{eqn:bid}, begin with the following differential Bianchi identity:
\beqa
\label{eqn:bianchi}
(\cd{\kk}{R})(\kk,\mm,\mm,\mb) + (\cd{\mm}{R})(\kk,\mm,\mb,\kk) + (\cd{\mb}{R})(\kk,\mm,\kk,\mm)\ =\ 0.
\eeqa
Expanding the first term, we obtain
\beqa
&&\hspace{-.1in}(\cd{\kk}{R})(\kk,\mm,\mm,\mb)\ =\ \kk[R(\kk,\mm,\mm,\mb)]\,-\, R(\cd{\kk}{\kk},\mm,\mm,\mb)\nonumber\\
&&\hspace{.1in}\,-\, R(\kk,\cd{\kk}{\mm},\mm,\mb)\,-\, R(\kk,\mm,\cd{\kk}{\mm},\mb)\,-\, R(\kk,\mm,\mm,\cd{\kk}{\mb}),\nonumber
\eeqa
whose terms on the right-hand side further simplify, via Lemma \ref{lemma:cd}:
\beqa
\kk(R(\kk,\mm,\mm,\mb)) &=& -\kk[{\rm Ric}(\kk,\mm)],\phantom{\underbrace{R}_{0}}\nonumber\\
R(\cd{\kk}{\kk},\mm,\mm,\mb) &=& -\bar{\kappa}\,\underbrace{R(\mm,\mm,\mm,\mb)}_{0}\,-\,\kappa\!\!\!\underbrace{R(\mb,\mm,\mm,\mb)}_{-\text{Ric}(\mm,\mb) + \frac{1}{2}\text{Ric}(\kk,\kk)},\nonumber\\ 
R(\kk,\cd{\kk}{\mm},\mm,\mb) &=& \kappa\,\underbrace{R(\kk,\kk,\mm,\mb)}_{0}\ +\ \vep\underbrace{R(\kk,\mm,\mm,\mb)}_{-\text{Ric}(\kk,\mm)},\nonumber\\
R(\kk,\mm,\cd{\kk}{\mm},\mb) &=& \kappa\underbrace{R(\kk,\mm,\kk,\mb)}_{-\frac{1}{2}\text{Ric}(\kk,\kk)}\ +\ \vep\underbrace{R(\kk,\mm,\mm,\mb)}_{-\text{Ric}(\kk,\mm)},\nonumber\\
R(\kk,\mm,\mm,\cd{\kk}{\mb}) &=& \bar{\kappa}\underbrace{R(\kk,\mm,\mm,\kk)}_{\text{Ric}(\mm,\mm)}\ +\ \bar{\vep}\underbrace{R(\kk,\mm,\mm,\mb)}_{-\text{Ric}(\kk,\mm)}.\nonumber
\eeqa
Thus the term $(\cd{\kk}{R})(\kk,\mm,\mm,\mb)$ is
$$
-\kk[{\rm Ric}(\kk,\mm)] + \kappa\,\text{Ric}(\kk,\kk) + \vep\,\text{Ric}(\kk,\mm) - \kappa\,\text{Ric}(\mm,\mb) - \bar{\kappa}\,\text{Ric}(\mm,\mm),
$$
where we have used the fact that $\vep + \bar{\vep} = 0$.  The remaining two terms in \eqref{eqn:bianchi} are simplified via the same (laborious) process, to yield \eqref{eqn:bid}.  Repeating this analysis on the differential Bianchi identity
$$
(\cd{\kk}{R})(\mb,\mm,\mm,\mb) + (\cd{\mm}{R})(\mb,\mm,\mb,\kk) + (\cd{\mb}{R})(\mb,\mm,\kk,\mm)\ =\ 0
$$
yields \eqref{eqn:bid2}.  (It is straightforward to verify that \eqref{eqn:bid} and \eqref{eqn:bid2} are the only nontrivial differential Bianchi identities.)
\end{proof}

Lemmas \ref{lemma:sachs1} and \ref{lemma:bid} exhibit one of the virtues of the Newman-Penrose formalism: the geometric quantities of interest pertaining to the flow\,---\,$\kappa,\rho,\sigma$\,---\,appear in \emph{first order} differential equations, in particular \eqref{eqn:S1}, \eqref{eqn:S2}, and \eqref{eqn:S3}.  We close this section with a lemma regarding rotations of the complex vectors $\mm,\mb$, analogous to \cite[Remark 5.8.4, p. 336]{o1995}.

\begin{lemma}
\label{lemma:boost}
Let $\{\kk,\mm,\mb\}$ be a complex triad.  There exists a smooth real function $\vartheta$ such that the modified complex triad
$$
\kk\hspace{.1in},\hspace{.1in}\mm_1\ :=\ e^{\ii\vartheta}\mm\hspace{.1in},\hspace{.1in}\mb_1\ :=\ e^{-\ii\vartheta}\mb
$$
has spin coefficients $\kappa_1 = e^{\ii\vartheta}\kappa, \sigma_1 = e^{2\ii\vartheta}\sigma, \rho_1 = \rho$, and $\vep_1 = 0$.
\end{lemma}

\begin{proof}
By definition, 
$$
\kappa_1\ =\ -\ip{\cd{\kk}{\kk}}{\mm_1}\ =\ e^{\ii\vartheta}\kappa;
$$
similarly, $\sigma_1 = e^{2\ii\vartheta}\sigma$, and $\rho_1 = \rho$.  Finally,
\beqa
\vep_1 &=& \ip{\cd{\kk}{\mm_1}}{\mb_1}\nonumber\\
&=& e^{-\ii \vartheta}\ip{\cd{\kk}{(e^{\ii\vartheta}\mm)}}{\mb}\nonumber\\
&=& \vep + e^{-\ii \vartheta}\kk[e^{\ii \vartheta}]\nonumber\\
&=& \vep + \ii\kk[\vartheta].\nonumber
\eeqa
Recalling that $\vep = \ii \ip{\cd{\kk}{\xx}}{\yy}$, we can choose any smooth real function $\vartheta$ satisfying $\kk[\vartheta] = -\ip{\cd{\kk}{\xx}}{\yy}$, in which case $\vep_1 = 0$.
\end{proof}

\section{The Newman-Penrose formalism and Ricci curvature}
\label{sec:sectional}
Now to the fruits of our labor.  First, a result for arbitrary Riemannian 3-manifolds.

\begin{prop}
\label{prop:ray}
Let $\kk$ be a smooth unit vector field with geodesic flow defined in an open subset of a Riemannian 3-manifold.  Let $p$ be a point in the domain of $\kk$ and $\gamma$ the geodesic integral curve of $\kk$ through $p$.  Then $\kk$ is hypersurface-orthogonal at $p$ if and only if it is hypersurface-orthogonal at every point along $\gamma$.
  Moreover, if $\gamma$ is complete, ${\rm Ric}(\kk,\kk) \circ \gamma \geq 0$, and $\kk$ is hypersurface-orthogonal along $\gamma$, then $({\rm div}\,\kk) \circ \gamma \geq 0$.
\end{prop}

\begin{proof}
The real and imaginary parts of \eqref{eqn:S1} in Lemma \ref{lemma:sachs1} are, respectively,
\beqa
\kk[\text{div}\,\kk] &=& \frac{\omega^2}{2} - 2|\sigma|^2 - \frac{(\text{div}\,\kk)^2}{2} - \text{Ric}(\kk,\kk),\label{eqn:ray}\\
\kk[\omega] &=& -(\text{div}\,\kk)\,\omega.\label{eqn:zero}
\eeqa
These are identical to their counterparts for a null geodesic flow in four-dimensional Lorentzian geometry (see \cite{sachs1961} and \cite[Prop. 5.7.2, p. 330]{o1995}); in particular, \eqref{eqn:ray} is the analogue of the \emph{Raychaudhuri equation} in general relativity \cite{ray1955}.  Hence it has the same application: if ${\rm Ric}(\kk,\kk) \circ \gamma \geq 0$ and $\kk$ is hypersurface-orthogonal along $\gamma$ (recall that by \eqref{eqn:rotation}, this is the case if and only if $\omega = 0$), then \eqref{eqn:ray} yields the inequality
$$
\kk[\text{div}\,\kk]\ \leq\ - \frac{(\text{div}\,\kk)^2}{2}
$$
along $\gamma$.  If $\text{div}\,\kk|_q < 0$ at any point $q$ along $\gamma$, then this inequality implies that $(\text{div}\,\kk)\circ \gamma \to -\infty$ at a finite value of the geodesic parameter, provided that $\gamma$ is defined at that value.  Since we are assuming that $\gamma$ is complete, this contradicts the smoothness of the function $(\text{div}\,\kk) \circ \gamma$.  Hence we must have $(\text{div}\,\kk) \circ \gamma \geq 0$.  Finally, \eqref{eqn:zero} implies that along any geodesic integral curve of $\kk$, either $\omega$ vanishes identically or else it is never zero.
\end{proof}

Bearing Lemma \ref{lemma:Killing} in mind, Proposition \ref{prop:ray} implies that if $\text{Ric}(\kk,\kk) < 0$, then $\kk$ cannot be a (unit length) Killing vector field\,---\,a fact which, as we mentioned in Section \ref{sec:review} above, is true in all dimensions (of course, there is also the classical result of Bochner \cite{Bochner}, namely, that \emph{compact} Riemannian manifolds with $\text{Ric} < 0$ have no nontrivial Killing vector fields).  More interesting for us, therefore, is the case of $\text{Ric}(\kk,\kk) > 0$.  When $M$ is compact, we can recover a result discovered in  \cite{hp13}.

\begin{cor}
\label{cor:hp}
Let $M$ be a closed Riemannian 3-manifold and $\kk$ a smooth, globally defined unit vector field with geodesic flow.  If \emph{$\text{Ric}(\kk,\kk) > 0$}, then $\kk$ is not hypersurface-orthogonal at any point of $M$.
\end{cor}

\begin{proof}
Since $M$ is compact, the flow of $\kk$ is complete and $\text{Ric}(\kk,\kk) \geq b$ for some constant $b > 0$.  Now consider any geodesic integral curve $\gamma$ of $\kk$ and suppose that $\kk$ is hypersurface-orthogonal along $\gamma$, hence that $\omega = 0$ everywhere along $\gamma$.   Then along this geodesic \eqref{eqn:ray} reduces to
$$
-\kk[\text{div}\,\kk]\ =\ 2|\sigma|^2 + \frac{(\text{div}\,\kk)^2}{2} + \text{Ric}(\kk,\kk)\ \geq\ b,
$$
hence $\kk[\text{div}\,\kk] \circ \gamma \leq -b$.  But this contradicts the result arrived at in Proposition \ref{prop:ray}, namely, that $({\rm div}\,\kk) \circ \gamma \geq 0$.  Hence $\omega$ cannot be identically zero along $\gamma$, which by Proposition \ref{prop:ray} is tantamount to it being everywhere nonzero along $\gamma$.  As $\gamma$ was chosen arbitrarily, the proof is complete.
\end{proof}

The result in \cite{hp13} appears in a slightly different guise.  Specifically, it was shown in \cite{hp13} that if a closed and orientable Riemannian 3-manifold has a $\kk$ satisfying the conditions in Corollary \ref{cor:hp}, then $\ip{\kk}{\cdot}$ is a contact form and $\kk$ its Reeb vector field.  But since $\ip{\kk}{\cdot}$ is a contact form at $p \in M$ (in other words, $d\ip{\kk}{\cdot}|_{\kk_p^{\perp}}$ is nondegenerate) if and only if $\ip{\kk}{[\xx,\yy]}|_p \neq 0$, which, by Frobenius's theorem, is the case if and only if $\omega_p \neq 0$, it follows that the result in \cite{hp13} and Corollary \ref{cor:hp} are equivalent.  (Passing to the smooth orientation covering of $M$, the result in \cite{hp13} of course also holds for non-orientable $M$.)  As also discussed in \cite{hp13}, a simple application of Corollary \ref{cor:hp} is to the round 3-sphere: in \cite{gluck2001} it was shown that all divergence-free great circle flows globally defined on the round 3-sphere must be tangent to a Hopf fibration.  By Corollary \ref{cor:hp}, those that are unit speed geodesics can never be hypersurface-orthogonal (not even at just one point).  As this does not require compactness, we mention it as a corollary.

\begin{cor}
\label{cor:hopf}
Let $M$ be a Riemannian 3-manifold and $\kk$ a smooth unit vector field defined in an open subset of $M$ whose flow is geodesic and divergence-free.  If ${\rm Ric}(\kk,\kk) > 0$, then $\kk$ is not hypersurface-orthogonal at any point in its domain.
\end{cor}

\begin{proof}
When $\text{div}\,\kk = 0$, \eqref{eqn:ray} reduces to
$$
\text{Ric}(\kk,\kk)\ =\ \frac{\omega^2}{2}\,-\,2|\sigma|^2,
$$
from which the statement immediately derives.
\end{proof}

Moving on, note that thus far we have not made use of the differential Bianchi identities \eqref{eqn:bid} and \eqref{eqn:bid2}.  As these are complicated equations, consider for the moment the case when the Riemannian 3-manifold has constant sectional curvature $K$.  Given a complex triad $\{\kk,\mm,\mb\}$ on such a manifold, its Ricci tensor satisfies
$$
\text{Ric}(\mm,\mm)\ =\ \text{Ric}(\kk,\mm)\ =\ 0\hspace{.1in},\hspace{.1in}\text{Ric}(\mm,\mb)\ =\ \text{Ric}(\kk,\kk)\ =\ 2K.
$$
(Recall that $\text{Ric} = 2K\langle\,,\rangle$.)  However, it is easy to verify that neither of the differential Bianchi identities \eqref{eqn:bid} and \eqref{eqn:bid2} gives any information regarding the flow of $\kk$ on such a manifold\,---\,as far as the differential Bianchi identities are concerned, constant sectional curvature is simply too stringent a condition.  But as we now show, \eqref{eqn:bid} and \eqref{eqn:bid2} certainly do give useful information about the \emph{scalar} curvature of the manifold.

\section{The Newman-Penrose formalism and scalar curvature}
\label{sec:GS}

The result we will prove in this section is akin to the \emph{Goldberg-Sachs theorem} from four-dimensional Lorentzian geometry (see \cite{gsrepub}, \cite{newpen62}, and \cite[Theorem 5.10.1, p. 345]{o1995}).  The Goldberg-Sachs theorem states that for any Ricci-flat four-dimensional Lorentzian manifold $M$, the flow of a smooth null vector field $\kk$ is geodesic and shear-free if and only if the Weyl tensor $C$ of $M$ satisfies at least one of the following conditions: $C(\kk,\xx,\kk,\cdot) = 0, C(\kk,\vv{x},\cdot,\cdot) = 0$, or $C(\kk,\cdot,\cdot,\cdot) = 0$, for all $\vv{x} \perp \kk$.  Such vector fields $\kk$ are known as \emph{repeated-principal} null vectors (see \cite[Def. 5.5.2, p. 318, Prop. 5.5.5, p. 321]{o1995}).  We would like to see if a relation such as this holds on Riemannian 3-manifolds (with the full Riemann tensor replacing the Weyl tensor, which vanishes in three dimensions).  Certainly we are motivated to begin with the following definition.

\begin{defn}
\label{defn:principal}
Let $M$ be a non-flat Riemannian 3-manifold with Riemann 4-tensor $R$.  A smooth unit vector field $\kk$ defined in an open subset of $M$ is {\rm 2-principal} if $R(\kk,\cdot,\cdot,\cdot) = 0$.
\end{defn}

We have already seen examples of 2-principal vector fields: any hypersurface-orthogonal unit Killing vector field $\kk$ is 2-principal, because such a $\kk$ is determined by the condition $\kappa = \rho = \sigma = 0$, from which, using the Sachs equations \eqref{eqn:S1}-\eqref{eqn:S3} and $(*)$, it follows easily that $R(\kk,\cdot,\cdot,\cdot) = 0$.  (Alternatively, inserting $\kappa = \rho = \sigma = 0$ into \eqref{eqn:S1} yields $\text{Ric}(\kk,\kk) = 0$; but any constant length Killing vector field $\kk$ satisfying $\text{Ric}(\kk,\kk) = 0$ must have vanishing covariant differential, $\nabla\kk = 0$ (see, e.g., \cite{BN}), and this easily implies that it must be 2-principal.)  Evidently, the condition of being 2-principal is quite stringent.  For example, on a Riemannian 3-manifold with strictly positive or strictly negative Ricci curvature, or with constant nonzero sectional curvature, there can be no 2-principal vector fields.  But as we now show, when 2-principal vector fields are present they constrain the scalar curvature of the manifold.  The differential Bianchi identities \eqref{eqn:bid} and \eqref{eqn:bid2} play a crucial here.

\begin{prop}
\label{prop:GS0}
Let $\kk$ be a smooth unit vector field defined in an open subset of a Riemannian 3-manifold.  If $\kk$ is 2-principal, $R(\kk,\cdot,\cdot,\cdot) = 0$, then the scalar curvature $S$ satisfies
\beqa
\label{eqn:S}
\kk[S]\ =\ -({\rm div}\,\kk)\,S.
\eeqa
Moreover, if $S$ is nonzero at a point $p$ in the domain of $\kk$, then the integral curve of $\kk$ through $p$ is a geodesic.
\end{prop}

\begin{proof}
Let $\{\kk,\mm,\mb\}$ be a complex triad.  Because $\kk$ is 2-principal,
$$
\text{Ric}(\mm,\mm)\ =\ \text{Ric}(\kk,\kk)\ =\ \text{Ric}(\kk,\mm)\ =\ 0\hspace{.1in},\hspace{.1in}\text{Ric}(\mm,\mb) = \frac{S}{2}\nonumber
$$
(recall $(*)$).  Inserting these into the second differential Bianchi identity \eqref{eqn:bid2} yields
$$
\kk[S]\ =\ (\rho + \bar{\rho})\,S.
$$
As $\rho + \bar{\rho} = -\text{div}\,\kk$, this is \eqref{eqn:S}.  Similarly, the first differential Bianchi identity \eqref{eqn:bid} yields
\beqa
\label{eqn:Sgeod}
\kappa\,S\ =\ 0.
\eeqa
As with \eqref{eqn:zero} in Proposition \ref{prop:ray}, the evolution equation \eqref{eqn:S} implies that if $S$ is nonzero at a point $p$ in the domain of $\kk$, then it is nonzero along the integral curve of $\kk$ through $p$.  Thus \eqref{eqn:Sgeod} dictates that $\kappa = 0$ along this curve, hence it must be a geodesic.
\end{proof}

There is one further property of 2-principal vector fields that we report here.  Inserting $\kappa = \rho + \bar{\rho} = \text{Ric}(\kk,\kk) = 0$ into \eqref{eqn:S1}, we obtain
\beqa
\label{eqn:0det}
|\sigma|^2\ =\ \frac{\omega^2}{4}\cdot
\eeqa
As the determinant of the matrix \eqref{eqn:matrix} is $\omega^2/4 + (\text{div}\,\kk)^2/4 - |\sigma|^2$, it follows that the linear map $\vv{v} \in \kk^{\perp} \mapsto \cd{\vv{v}}{\kk}$ has zero determinant everywhere (it also follows, via \eqref{eqn:S1} and \eqref{eqn:S2}, that $\kk[\omega] = \kk[\sigma] = 0$).  Armed with \eqref{eqn:0det}, we then have the following variant of the Goldberg-Sachs theorem, with which we close this paper.

\begin{cor}
\label{cor:GS0}
Let $\kk$ be a smooth, hypersurface-orthogonal unit vector field in an open subset of a Riemannian 3-manifold with constant nonzero scalar curvature.  Then $\kk$ is 2-principal, $R(\kk,\cdot,\cdot,\cdot) = 0$, if and only if $\kk$ is a Killing vector field.
\end{cor}

\begin{proof}
As discussed above, if $\kk$ is a hypersurface-orthogonal unit Killing vector field, then $\kk$ is 2-principal.  The converse follows from Proposition \ref{prop:GS0} and \eqref{eqn:0det}.
\end{proof}

\bibliographystyle{siam}
\bibliography{NP-3Riem}
\end{document}